\documentclass{amsart}

\usepackage{latexsym,amsfonts}
\usepackage{amssymb}
\usepackage{amsmath}
\usepackage{amsthm}

\def\SO{\mathop{\rm SO}\nolimits}
\def\SL{\mathop{\rm SL}\nolimits}
            
\def\PSL{\mathop{\rm PSL}\nolimits}          
\def\PSp{\mathop{\rm PSp}\nolimits}           
\def\PGL{\mathop{\rm PGL}\nolimits}           
\def\SU{\mathop{\rm SU}\nolimits}          
\def\PSU{\mathop{\rm PSU}\nolimits}          
\def\Sp{\mathop{\rm Sp}\nolimits}

\def\dim{\ \mathop{\rm dim}\nolimits\ }

\def\diag{\mathop{\rm diag}\nolimits}

\newtheorem{theorem}{Theorem}[section] 
\newtheorem{lemma}[theorem]{Lemma}     

\newtheorem{proposition}[theorem]{Proposition}
\newtheorem{remark}[theorem]{Remark}

\newtheorem*{theo}{Theorem}

\newcommand{\Z}{\mathbb{Z}}    
\newcommand{\F}{\mathbb{F}}    

\begin{document}

\title[The $(2,3)$-generation of the classical simple groups]{The $(2,3)$-generation of the classical simple \\ groups of dimension $6$ and $7$}

\author{Marco Antonio Pellegrini}
\address{Dipartimento di Matematica e Fisica, Universit\`a Cattolica del Sacro Cuore, Via Musei 41,
I-25121 Brescia, Italy}
\email{marcoantonio.pellegrini@unicatt.it}

\dedicatory{Dedicated to M. C. Tamburini on the occasion of her 70th birthday}

\begin{abstract}
In this paper we prove that the finite simple groups $\PSp_6(q)$,  $\Omega_7(q)$ and  $\PSU_7(q^ 2)$ are $(2,3)$-generated
for all $q$. In particular, this result  completes the classification of the $(2,3)$-generated finite classical simple groups up to dimension $7$.
\end{abstract}

\maketitle

\section{Introduction}

Given a finite group, it is very  natural to ask for a minimal set of generators. For finite non-abelian simple groups it is well known that they can be generated by a pair $X,Y$ of suitable elements: for alternating groups this is an old result of Miller \cite{Mill}, for groups of Lie type it is due to Steinberg \cite{Stein} and for sporadic groups to Aschbacher and Guralnick \cite{AG}.

Since non-abelian simple groups have even order, on many occasions group theorists require that one of these two elements, say $X$, is an involution. It was conjectured (for example, see \cite{DT}) that every finite non-abelian simple group can be generated by an involution $X$ and an element $Y$ of order $\geq 3$. For alternating and sporadic groups the answer was already provided in \cite{Mill} and \cite{AG}, respectively. Groups of Lie type attracted the interest of many authors: among them,  Malle, Saxl and Weigel proved in \cite{MSW}  the validity of the conjecture (at least for $G\neq \PSU_3(9)$), taking $Y$ to be strongly real. Their result clearly implies that every finite non-abelian simple group different from $\PSU_3(9)$ is generated by a set of three involutions. For $G=\PSU_3(9)$ a direct computations shows that this group can be generated by an involution and an element of order $7$.

Considering simple groups of Lie type, one can require, for instance, that $Y$ is contained in few maximal subgroups, as done for instance in \cite{GK}, or that the order of $Y$ is a prime.
In particular we say that a group is  $(2,p)$-generated if it can be generated by two elements $X,Y$ of respective orders $2$ and $p$, where $p$ is a prime.
Since groups generated by two involutions are dihedral, we must have $p\geq 3$.  It seems that the difficulties increase when one investigates the $(2,p)$-generation for some fixed small prime $p$. The choice $p=3$ is the most natural, since $(2,3)$-generated groups are, along with the groups of order at most $3$, the homomorphic images of $\PSL_2(\Z)$. 

A key result for this kind of problem is due to Liebeck and Shalev who proved in \cite{LS} that, apart from the infinite families $\PSp_4(2^m)$, $\PSp_4(3^ m)$ and ${}^2B_2(2^{2m+1})$,  all finite non-abelian simple groups are $(2,3)$-generated with a finite number of exceptions. 
However, their result relies on probabilistic methods and does not provide any estimates on the number or the distribution of such exceptions.
The exceptional groups of Lie type were studied by L\"{u}beck and Malle \cite{LM} and so
the problem of  finding the exact list of simple groups which are not $(2,3)$-generated reduces to the classical groups where it is still wide open (see Problem 18.98 in \cite{MK}). 
In view of papers such as \cite{TGL,TW,TWG}, we have to consider only groups of small dimension, say, less than $13$ for $\PSL_n(q)$ and less than $88$ for the other classical groups.

In \cite{M}, Macbeath dealt with the groups $\PSL_2(q)$. For classical groups of dimension $3$ and $5$ the groups
which are not $(2,3)$-generated are described in \cite{35} and for dimension $4$ in \cite{PTV}. 

With this paper we conclude the classification of the $(2,3)$-generated finite classical simple  groups up to dimension $7$. The $(2,3)$-generation of $\SL_6(q)$, $\SL_7(q)$ and $\SU_6(q^ 2)$ (and their projective images) for all $q$ was proved in \cite{T6}, \cite{T7} and \cite{PPT}, respectively. When $q$ is odd, the symplectic groups $\Sp_6(q)$ are not $(2,3)$-generated (see \cite{V}), but the simple groups $\PSp_6(q)$ are $(2,3,7)$-generated for $q>3$ 
(see \cite{TV6}) and $\PSp_6(3)$ is $(2,3)$-generated, for instance, by the projective images of the following matrices: 
$$x=\left(\begin{array}{cccccc}
-1 & 0 & 0 & 1 & 0 & 0\\
0 & -1 & 0 & 0 & 1 & 0\\
0 & 0 & 0 & 0 & 0 & 1\\
1 & 0 & 0 & 1 & 0 & 0\\
0 & 1 & 0 & 0 & 1 & 0\\
0 & 0 & -1 & 0 & 0 & 0\\
\end{array}\right),\quad
y=\left(\begin{array}{cccccc}
0 & 0 & 1 & 1 & 1 & 1\\
1 & 0 & 0 & 1 & 1 & 1\\
0 & 1 & 0 & 1 & 1 & 1\\
0 & 0 & 0 & 0 & 0 & 1\\
0 & 0 & 0 & 1 & 0 & 0\\
0 & 0 & 0 & 0 & 1 & 0
\end{array}
\right).$$
The remaining cases,  studied in this paper, are the groups $\Sp_6(q)$ for $q$ even,
$\Omega_7(q)$ for $q$ odd and $\SU_7(q^ 2)$ for all $q$ (our notation for the classical
groups accords with  \cite{Car}).
Observe that among these classes of groups only $\PSU_7(p^{n_7})$, for any prime $p\neq
7$ such that  its order $n_7$ modulo $7$ is even, are $(2,3,7)$-generated, see \cite{PTG2,
TV1}.

From Theorems \ref{main6}, \ref{main7}, \ref{mainOm7} and Propositions \ref{U72} and \ref{Om35} below, we deduce the following.

\begin{theo}
The finite simple groups $\PSp_6(q)$, $\Omega_7(q)$ and $\PSU_7(q^ 2)$ are $(2,3)$-gene\-ra\-ted for all $q$.
\end{theo}

We point out that there are only $8$ exceptions for simple classical groups of dimension $\leq 7$. 
To our knowledge the only exceptions in higher dimension are $\Omega^+_8(2)$ and P$\Omega^+_8(3)$ (M. Vsemirnov, 2012). 
Furthermore, our generators were constructed in a uniform way starting from the generators provided in \cite{35} and following the inductive method described in \cite{PPT}. The idea is to construct the $(2,3)$-generators of dimension $n$ from those of dimension $n-1$. In fact, dimensions $6$ and $7$ can be considered the basis of this induction process for what concerns unitary, symplectic and odd dimensional orthogonal groups in any characteristic.
\smallskip

Throughout this paper, for any fixed power $q$ of the prime $p$,  $\F_{q}$ denotes the finite field of $q$ elements and $\F$ is its algebraic closure.  Furthermore, the set $\{e_1,e_2,\ldots,e_n\}$ is the canonical basis of
$\F^n$. For a description of the maximal subgroups of the low-dimensional classical groups
we refer to \cite{Ho}. Finally, $\omega$ is an element of $\F$ of order
$3$ if $p\neq 3$, $\omega=1$ otherwise.

 \section{The groups $\Sp(6,q)$, $q$ even}
 
In order to prove the $(2,3)$-generation of the groups $\Sp_6(q)$ for $q$ even, we first construct our pair $x,y$ of $(2,3)$-generators. Take $a\in \F_q^\ast$, with $q$ even, and define $H=\langle x,y\rangle$ where
\begin{equation}\label{genSp}
x=\left(\begin{array}{cccccc} 
 0 &  0 &  1 &  0 &  0 &  0 \\
 0 &  0 &  0 &  1 &  0 &  0 \\
 1 &  0 &  0 &  0 &  0 &  0 \\
 0 &  1 &  0 &  0 &  0 &  0 \\
 0 &  0 &  0 &  0 &  1 &  0 \\ 
 0 &  0 &  0 &  0 &  a &  1 
 \end{array}\right), \quad
y=\left(\begin{array}{ccccccc} 
 1 &  0 &  0 &  1 &  1 &  0 \\
 0 &  1 &  0 &  0 &  0 &  0 \\
 0 &  0 &  0 &  1 &  0 &  0 \\
 0 &  0 &  1 &  1 &  0 &  0 \\
 0 &  a &  1 &  1 &  1 &  1 \\
 0 &  a &  1 &  0 &  1 &  0 
\end{array}\right) .
 \end{equation}
Observe that $x^ 2=y^ 3=I$ and that $H\leq \Sp_6(q)$ since they fix the
following Gram matrix (that is, $x^ TJx=y^ TJy=J=J^T$):
$$J=\left(\begin{array}{ccccccc}
0 & 1 & 0 & 0 & 0 &0\\
1 & 0 & 0 & a & a+1 &  1\\
0 & 0 & 0 & 1 & 0   & 0 \\
0 & a & 1 & 0 & 1 & 1\\ 
0 & a+1 &  0 & 1 & 0 &1 \\
0 & 1 & 0 & 1 & 1 & 0 
     \end{array}\right).$$
 
The invariant factors of $x$ and $y$ are, respectively,
\begin{equation}\label{inv6}
(t^2+1), \ (t^2+1),\ (t^2+1)\quad\mbox{and}\quad (t^3+1),\ (t^3+1). 
\end{equation}
The minimum polynomial of $z=xy$ coincides with its characteristic polynomial
\begin{equation}\label{char6}
\chi_{z}(t)=t^6+(a+1)t^5 +t^4+ t^3+ t^2+(a+1)t+1.
\end{equation}

\begin{proposition}\label{irridSp}
The subgroup $H$ is absolutely irreducible.
\end{proposition}

\begin{proof}
It suffices to prove the irreducibility of  $H$ viewed as a subgroup of $\Sp_6(\F)$ where $\F$ denotes the algebraic closure of $\F_q$ (so the parameter $a$ in \eqref{genSp} is an element of $\F^\ast$).
Suppose that $W\neq V$ is an $H$-invariant subspace of $V=\F^6$. 
Observe firstly that  for all $v \in V$ the vector $v+yv+y^2 v$ always belongs to the subspace 
$U=\langle e_1, e_2+a e_6\rangle$.  
So assume that, for some $w\in W$,  we have  $u=w+yw+y^2w \neq 0$ (hence $0\neq u\in W\cap U$).
We will show that this implies the absurd $W=V$.

Let $u=(x_1, x_2, 0,0,0,ax_2)^ T\neq 0$. Replacing, if necessary, $w$ by a suitable scalar 
multiple, without loss of generality, we may assume that either $(x_1,x_2)=(1,0)$ or $x_2=1$.
If $(x_1,x_2)=(1,0)$, then the matrix 
$M_1$, whose 
columns are
$$M_1  =\left(u  \mid xu \mid yxu  \mid xyxu \mid  (yx)^ 2u \mid x(yx)^ 2u \right),$$
has determinant $a\neq 0$ and hence the $H$-submodule generated by $u$ is the whole space $V$, 
contradicting the fact that it is contained in the proper subspace $W$.
If $x_2=1$, then, the matrix $M_2$, whose columns are
$$M_2  =\left(u  \mid xu \mid yxu  \mid xyx u \mid  xy^ 2x u \mid y^ 2 xu \right),$$
has determinant  $a^ 3(x_1+a+1)^ 2$ which is non-zero if $x_1\neq a+1$. In this case, as before, 
the $H$-submodule generated by $u$ is the whole space $V$, producing the contradiction $V\leq W$.
So assume further that $x_1=a+1$. In this case,
if $a$ has order $3$ we obtain the absurd $V=\langle u, xu, yxu, xy^ 2xu, (yx)^ 2u,$ $ (xy)^ 2yxu 
\rangle\leq W$  and for the other choices of $a\neq 0$ we get the absurd $V=\langle u, xu, yxu, y^ 
2xu, xy^ 2x u, yxy^ 2xu \rangle\leq W$.

Hence, we proved that  $W\cap U=\{0\}$ or, in other words, that all the elements $w$ of $W$ satisfy 
the condition
\begin{equation}\label{irrr}
w+yw+y^ 2w=0 
\end{equation}
It follows that every element $w$ of $W$ has shape $w=(x_1, 0, x_3, x_4, x_5, x_1+x_5)^ T$. 
Fix a such $w\in W$. Since also $xw\in W$ satisfies condition \eqref{irrr},
we obtain $x_4=0$ and $x_3=x_1+ax_5$. Next, considering $xyxw\in W$  we get $x_1=0$ and 
$(a^2+a+1)x_5=0$. Finally 
applying condition \eqref{irrr} to $(xy)^ 2 x w\in W$ we obtain $(a+1)x_5=0$, whence $x_5=0$ and so 
$w=0$ (i.e., $W=\{0\}$).
\end{proof}

Now, we want to analyze when $H$ is contained in a maximal subgroup of $\Sp_6(q)$. For a description of these subgroups, see \cite[Tables 8.28 and 8.29]{Ho}.

\begin{lemma}\label{monomialSp}
The subgroup $H$ is not monomial.
\end{lemma}

\begin{proof}
Let $\mathcal{B}=\{v_1, \dots\}$ be a basis on which $H$ acts monomially.
Considering transitivity and canonical forms \eqref{inv6}, the permutation induced
by $x$ is the product of three  $2$-cycles. Furthermore, we may assume 
$$v_2=yv_1, \ v_3=yv_2, \ v_4=xv_1, \ v_5=yv_4, \ v_6=yv_5.$$
We have the following cases.
\begin{itemize}
\item[(i)] $xv_2=\lambda v_5$ and $xv_3=\mu v_6$. In this case, $\chi_{xy}(t)=t^6+1$, in
contradiction with \eqref{char6}.
\item[(ii)] $xv_2=\lambda v_6$ and $xv_3=\mu v_5$. In this case $\chi_{xy}(t)=
t^6+ \frac{\lambda^2\mu+\mu^2+\lambda}{\lambda \mu} t^4
+\frac{\lambda \mu^2+\lambda^2+\mu}{\lambda \mu} t^2+1$.
Comparison with \eqref{char6} gives a contradiction.
\item[(iii)] $xv_2=\lambda v_3$ and $xv_5=\mu v_6$. In this case $\chi_{xy}(t)=
t^6+\frac{\lambda+\mu}{\lambda \mu} t^5+ \frac{1}{\lambda \mu} t^4+ (\lambda\mu)t^2+
(\lambda+\mu) t +1$. Once again, we obtain a contradiction with \eqref{char6}.
\end{itemize}
\end{proof}

\begin{lemma}\label{SOSp}
The subgroup $H$ is not contained in an orthogonal group
$\SO_6^\pm (q)$.
\end{lemma}

\begin{proof}
Suppose that $H\leq \SO_6^\pm(q)$ and let $Q$ be a quadratic form fixed by $H$.
This means that for all $v\in \F_q^ 6$, $Q(xv)=Q(v)$ and $Q(yv)=Q(v)$.
Recall further that $Q(v_1+v_2)+Q(v_1)+Q(v_2)=v_1^T J v_2$. 

From  $Q(e_5)=Q(xe_5)=Q(e_5+a e_6)$,  $Q(e_6)=Q(ye_6)=Q(e_5)$ and $Q(e_1)=Q(xe_1)=Q(e_3)$  we get respectively $Q(e_6)=1/a$, $Q(e_5+e_6)=1$ and $Q(e_1+e_3)=0$.
Now,  $Q(e_4)=Q(ye_4)=Q((e_1+e_3)+(e_4+e_5))$ implies $Q(e_4)=Q(e_1+e_3)+Q(e_4+e_5)+1=Q(e_4)+Q(e_5)+1+1$, whence $0=Q(e_5)=Q(e_6)=1/a$, a contradiction.
\end{proof}

\begin{lemma}\label{G2Sp}
The subgroup $H$ is not contained in a subgroup $M$ isomorphic
to $G_2(q)$.
\end{lemma}

\begin{proof}
We first observe that the  discriminant of the characteristic polynomial of $z=xy$ is $1$ and so
the eigenvalues of $z$ are pairwise distinct. Suppose that $H$ is contained in a subgroup $M\cong G_2(q)$.
We may embed $H$ in $\mathfrak{G}=G_2(\F)$ and by the previous observation $z$ is a semisimple element of
$\mathfrak{G}$, that is, it belongs to a maximal torus  of $\mathfrak{G}$. 
By  \cite{G2}, $z$ is conjugate to $s=\diag(\alpha,\beta,\alpha \beta,\alpha^ {-1},\beta^
{-1}, (\alpha\beta)^{-1})$ where $\alpha,\beta\in \F^\ast$.
Let $\chi_s(t)=t^ 6+\sum_{j=1}^ 5 f_j t^ j+1$ be the characteristic polynomial of $s$.
It is easy to see that $f_3=f_1^ 2$. Comparison with \eqref{char6} gives
$1=(a+1)^ 2$, that is $a=0$.
\end{proof}

\begin{theorem}\label{main6}
Take $x,y$ as in \eqref{genSp}, with $a\in \F_{q}^\ast$, $q$ even, such that
$\F_p[a]=\F_q$. Then $H=\langle x,y\rangle=\Sp_6(q)$. 
In  particular, the  groups $\Sp_6(q)$  are $(2,3)$-generated for all $q$.
\end{theorem}

\begin{proof}
We already observed that $H\le \Sp_6(q)$. 
 Let $M$ be a maximal subgroup of $\Sp_6(q)$ which contains $H$. 
By Proposition \ref{irridSp}, $H$  is absolutely irreducible and so  $M\not\in\mathcal{C}_1\cup\mathcal{C}_3$.
Moreover $M$ cannot belong to $\mathcal{C}_2$  by  Lemma \ref{monomialSp}.
Since $a+1$ is  a coefficient of the characteristic polynomial of $z$, the assumption
$\F_p[a]=\F_{q}$ implies that the matrix $z$
cannot be conjugate 
to  elements of $\Sp_6(q_0)$ for any $q_0<q$. Thus $M\not\in \mathcal{C}_5$. 
By Lemmas \ref{SOSp} and  \ref{G2Sp}, $M$ cannot be in  $\mathcal{C}_8\cup \mathcal{S}$.
We conclude that $H=\Sp_6(q)$.
\end{proof}

\section{The $7$-dimensional classical groups}

In order to prove that the groups $\SU_7(q^ 2)$ and $\Omega_7(q)$ are $(2,3)$-generated, 
we provide  a pair of uniform generators. Consider the following matrices:
\begin{equation}\label{gen7}
x=\begin{pmatrix}
0& 1& 0& 0& 0& 0& a\\
   1& 0& 0& 0& 0& 0& a \\
   0& 0& 0& 1& 0& 0& 0 \\
   0& 0& 1& 0& 0& 0& 0 \\
   0& 0& 0& 0& 0& 1& -1 \\
   0& 0& 0& 0& 1& 0& -1\\
   0& 0& 0& 0& 0& 0& -1 
\end{pmatrix},\quad 
y=\begin{pmatrix}
1& 0& -1& 0& -1& 0& a+b-1 \\
 0& 0& -1& 0& 0& 0& 0 \\
     0& 1& -1& 0& 0& 0& 0 \\
      0& 0& 0& 0& -1& 0& 0 \\
     0& 0& 0& 1& -1& 0& 0 \\
     0& 0& 0& 0& 0& 0& -1 \\
      0& 0& 0& 0& 0& 1& -1
\end{pmatrix},
\end{equation}
where either
\begin{equation}\label{O}
b=a\in \F_q \quad \textrm{ and } \quad H\leq \SL_7(q)
\end{equation}
or
\begin{equation}\label{U}
b=a^ q\in \F_{q^ 2} \quad \textrm{ and } \quad H\leq \SL_7(q^ 2).
\end{equation}

The invariant factors of $x$ and $y$ are respectively
\begin{equation}\label{inv7}
(t+1),\ (t^2-1),\ (t^2-1),\ (t^2-1);\qquad (t^2+t+1),\ (t^2+t+1),\ (t^3-1),
\end{equation}
and the characteristic polynomial of $z=xy$ is
\begin{equation}\label{char7}
\chi_{xy}(t)=t^7-t^5+(1-a)t^ 4+(b-1)t^3+t^2-1. 
\end{equation}

\begin{lemma}\label{unitary7}
If $H$ is absolutely irreducible, then  the characteristic
polynomial $\chi_z(t)$ of $z$ coincides with its  minimum polynomial. Furthermore,
under hypothesis \eqref{O} we have $H\leq \SO_7(q)$ and under hypothesis \eqref{U} we obtain $H\leq \SU_7(q^
2)$.
\end{lemma}

\begin{proof}
We have $\dim C(x)=25$, $\dim C(y)=19$ and $\dim C(z)=7$. From the Frobenius formula,
it follows that $z$ has a 
unique invariant factor, whence our first claim.
In particular, the triple $(x,y,z)$ is rigid (see \cite{PTV}) and by \cite[Theorem
3.1]{PTV} we obtain $H\leq \SO_7(q)$ when \eqref{O} holds and $H\leq \SU_7(q^2)$ when
\eqref{U} holds.
\end{proof}

\begin{proposition}\label{irr7}
Take $x,y$ as in \eqref{gen7}. Then the subgroup $H=\langle x,y \rangle$ is absolutely
irreducible if and only if the following conditions hold:
\begin{itemize}
\item[{\rm (i)}] $a^2-ab +b^ {2}+2a+2b+4=\prod_{j=1}^2 (b+\omega^j a-2\omega^{2j}) \neq 0$;
\item[{\rm (ii)}] $(a+b)^ 3 -8(a+b-2)^ 2 -8ab \ne  0$.
\end{itemize}
\end{proposition}

\begin{proof}
We consider the irreducibility of  $H$ viewed as a subgroup of $\SL_7(\F)$ (so the 
parameters $a,b$ in \eqref{gen7} are elements of $\F$).

First, assume that 
$b=-\omega^ja+2\omega^{2j}$ for some $j=1,2$ and consider the element
$w=\left(a+\omega^{2j},-\omega^{2j}, 1,-1, \omega^{j},\omega^{2j},-1\right)^T\neq 0$.
We have  $yw=\omega^{j}w$ and $xw=-w$. Thus $W=\langle w\rangle$ is a $1$-dimensional 
$H$-invariant subspace of $V=\F^7$.

Next, assume that 
$(a+b)^ 3 -8(a+b-2)^ 2 -8ab=0$.
If $p=2$, then $a=b$.  Taking $w=\left(1,1,1,1,1,1,0\right)^T$, 
the subspace $\langle w, yw, xyw, yxyw, (xy)^2w,
y(xy)^2w \rangle$ is $H$-invariant.
Assume $p\neq 2$. If $a=b=2$, then  consider
$w=\left(1,1,2,2,1,1,0  \right)^T$;
the subspace $\langle w, yw, xyw, yxyw \rangle $ is $H$-invariant.
If $(a,b)\neq (2,2)$, taking
$w=\left(x_1,x_1,x_2,x_2,x_3,x_3,0\right)^T$,
where $x_1=-\frac{(a+b)^2-6a-10b+16}{2}$, $x_2=
 2b-4$ and $x_3=  a+b-4$, we obtain that $w\neq 0$ and the subspace $\langle w,yw\rangle$ 
is $H$-invariant.

Now, assume that Conditions (i) and (ii) both hold and let $W\neq V$ be an $H$-invariant subspace 
of $V$. 
Straightforward calculation shows that, for all $v\in V$, the element
$v+yv+y^2v=(I+y+y^2)v$ always belongs to the subspace $\langle e_1\rangle$.
On the other hand, note that   
\begin{equation}\label{ei}
xe_1=e_2,\; ye_2=e_3, \; xe_3=e_4,\; ye_4=e_5,\; xe_5=e_6,\; ye_6=e_7. 
\end{equation}
It follows that if for some $w\in W$ we have $u=w+yw+y^2w\neq 0$, then the $H$-submodule generated 
by $u$ is the whole space $V$, in contradiction with the assumption $W\neq V$.
Hence, every element $w$ of $W$ 
satisfies the following condition: 
\begin{equation}\label{w}
w+yw+y^2w=0. 
\end{equation}
We will show that this condition implies $W=\{0\}$

\noindent \textbf{Case (a).} Suppose that $w+xw=0$ for all $w\in W$.
Then all vectors in $W$ have shape $( x_1, -a(x_5+x_6)-x_1, x_3, -x_3, x_5, x_6, x_5+x_6)^T$. 
We fix a non-zero $w\in W$. 
From  $yw+xyw=0$ and $y^2w+xy^2w=0$ we see that expression (i) must be $0$, a contradiction. 

\medskip

\noindent\textbf{Case (b).} There exists $\overline w\in W$ such that $\overline w+x\overline w\ne 0$.
Then $\overline w+x\overline w =w$ has shape $(x_1, x_1,x_3, x_3, x_5,x_5,0)^ T$.
Condition \eqref{w} gives $2x_3=2x_1+cx_5$, with $c=a+b-2$.
Suppose $p=2$.  If $c=0$ (that is $a=b$),  expression (ii) is $0$, a contradiction.
If $c\neq 0$, then $x_5=0$ and application of \eqref{w} to $xyw$ and to
$(xy)^2w$, which are in $W$, 
leads to $W=\{0\}$.
Thus $p\ne 2$, $x_3=x_1+\frac{c}{2}x_5$.
After this substitution, Condition \eqref{w} applied to $xyw$ gives
\begin{equation}\label{x1}
(c-2)x_1=\left(-\frac{1}{2}\left(a +b\right)^ 2 + 3a + 5b - 8\right)x_5.
\end{equation}
Assume $c= 2$ (that is $b=4-a$). Then $(a-2)x_5=0$.
If $a=2$ (and so $b=2$), then expression (ii) is $0$, a contradiction.
If $a\neq 2$, we get $x_5=0$.
In this case, applying \eqref{w} to $(xy)^2w\in W$, we obtain  $W=\{0\}$.
Assume  $c\neq 2$. Using \eqref{x1} to eliminate $x_1$ and applying \eqref{w} to $(xy)^2w$,
we see that expression (ii) must be $0$, a contradiction. 
\end{proof}

\begin{lemma}\label{diag}
Assume $H$ absolutely irreducible. If  $(xy)^k$ is  a diagonal matrix, then $k\geq 13$.
\end{lemma}

\begin{proof}
For $k\leq 3$ our claim follows from \eqref{ei}. Let $D(k)=(xy)^k$.  For $k=4,5,6$ we have
$D(k)_{3,2}=\pm 1$ and for $k=7$ we have $D(k)_{3,1}=1$.
If $k=8$, we obtain  $D(8)_{3,1}=a-2$, but for $a=2$ we get $D(8)_{7,2}=-1$.
For $k=9$, $D(9)_{7,1}=-2a+3$ and hence $p\neq 2$. 
In this case,  $a=\frac{3}{2}$ yields $D(9)_{7,5}=-1$.
For $k=10,11,12$,  we have $D(10)_{3,4}=D(11)_{3,2}=D(12)_{3,1}=(2-a)(a+2b-5)$. 
Assume first $a=2$. For $k=10,12$, we obtain $D(10)_{5,3}=D(12)_{6,4}=2-b$, whence
$b=2$. However, by Proposition \ref{irr7}(ii), $H$ is reducible.
For $k=11$ we have $D(11)_{7,1}=2(b-2)$. If $b=2$, as before, $H$ is reducible and if
$p=2$ we get $D(11)_{2,1}=1$.
Assume now $a=5-2b$. In this case, $D(10)_{3,5}=1-b$, however $b=1$ leads to $D(10)_{4,1}=7$ and $D(10)_{4,7}=5$. Moreover,  $D(11)_{3,4}=D(12)_{3,2}=(b-2)(9b-13)$. We can exclude $b=2$ (which implies $a=1$), since this produces the contradiction $D(11)_{1,4}=D(12)_{1,4}=1$. So, take $9b=13$ ($p\neq 3$). We obtain the contradiction $D(11)_{3,5}=D(12)_{5,7}=\frac{1}{9}$.
\end{proof}

We want now to prove that $H$, when absolutely irreducible,
is not contained in a maximal subgroup of $\Omega_7(q)$ or of $\SU_7(q^ 2)$.
We refer to  \cite[Tables 8.37, 8.38, 8.39 and 8.40]{Ho}.

\begin{lemma}\label{monomial7}
Assume $H$ absolutely irreducible. Then $H$ is not monomial.
\end{lemma}

\begin{proof}
Let $\{v_1, \ldots, v_7\}$ be a basis on which $H$ acts monomially and transitively.
We may then assume that $yv_1=v_2$, $yv_2=v_3$, $yv_3=v_1$, $yv_4=v_5$, $yv_5=v_6$,
 $yv_6=v_4$. However, this is in contradiction with the invariant factors \eqref{inv7} of $y$.
\end{proof}

\begin{lemma}\label{class S7}
Assume $H$ absolutely irreducible. Then $H$ is not conjugate to
any subgroup of type $Z\times \PSU_3(3)$, where $Z$ consists of scalar matrices. 
\end{lemma}

\begin{proof}
It suffices to notice that elements of $\PSU_3(3)$ have order $1,2,3,4$, $6,7,8$
or $12$ and apply Lemma \ref{diag}.
\end{proof}

\begin{lemma}\label{S62}
Assume hypothesis \eqref{O} and $H$ absolutely irreducible. Then $H$ is not contained in a subgroup $M$ isomorphic
to $\Sp_6(2)$.
\end{lemma}

\begin{proof}
Firstly, notice that the commutator between two elements of $\Sp_6(2)$ of order respectively $2$ and $3$ has order $7$.
So, if $H\leq M$, then $[x,y]$ must have order $7$.
Set $D=[x,y]^ 7$. Since we are assuming $b=a$, we obtain $D_{7,1}=(a-2)^ 2(1-a)$.
However, for $a=1,2$ the subgroup is reducible by Proposition \ref{irr7}(ii).
\end{proof}

We can now prove the $(2,3)$-generation of $\SU_7(q^ 2)$.

\begin{proposition}\label{U72}
The group $\SU_7(4)$ is $(2,3)$-generated. 
\end{proposition}

\begin{proof}
Take the following two matrices of $\SL_7(4)$:
$$x=\left(\begin{array}{ccccccc} 
0 &  1 &0&0&0&0&0\\
1 & 0 &0&0&0&0&0\\
0 & 0 &1&0&0&0&0\\
0 & 0 &0&1&1&0&\omega\\
0 &0 &0&1&0&\omega^2&\omega\\
0 & 0&0&0&\omega&1&\omega^2\\
0 &0&0&\omega^2&\omega^2&\omega&0
       \end{array}\right),\qquad
y=\left(\begin{array}{ccccccc} 
1&0&0&0&0&0&0\\
0&0&0&1&0&0&0\\
0&1&0&0&0&0&0\\
0&0&1&0&0&0&0\\
0&0&0&0&0&0&1\\
0&0&0&0&1&0&0\\
0&0&0&0&0&1&0
\end{array}\right).$$
Then $x^2=y^3=1$. Moreover, $x^ T x^ \sigma=y^ T y^ \sigma=I$ ($\sigma$ is the
automorphism of $\SL_7(q^2)$ defined as $(\alpha_{i,j})\mapsto (\alpha_{i,j}^q)$)
and so $H=\langle
x,y\rangle\leq \SU_7(4)$. 
Assume that $H$ is
contained in some maximal subgroup $M$ of $G$. Since $g=(xy^2xy)^ 2(xy^2)^ 3$ has order $43$, then $g$ can
be contained only in a maximal subgroup of
class $\mathcal{C}_3$: $M=\frac{2^7+1}{3}:7$. However,  $|M|=7\cdot 43$ and so $x \not \in
M$.
Hence, $H=\SU_7(4)$.
\end{proof}

\begin{theorem}\label{main7}
Take $x,y$ as in  \eqref{gen7} with $a\in \F_{q^2}\setminus \F_q$ and suppose that
\begin{itemize}
\item[{\rm (i)}]$ a^ {2q}-a^ {q+1}+a^2+2a^ q+2a+4\neq 0$;
\item[{\rm (ii)}] $(a+a^q)^3 -8(a+a^q-2)^2-8 a^{ q+1} \neq 0$, when $p$ is odd;
\item[{\rm(iii)}]  $\F_{q^2}=\F_p[a^7]$.
\end{itemize}
Then $H=\langle x,y\rangle=\SU_7(q^2)$. Moreover, if $q^2\neq 2^2$, then there exists
$a\in \F_{q^2}^\ast$ satisfying conditions {\rm (i)} to {\rm (iii)}. In particular, the
groups $\SU_7(q^2)$ and $\PSU_7(q^2)$ are $(2,3)$-generated for all $q$.
\end{theorem}

\begin{proof}
By Conditions (i) and (ii), $H$ is absolutely irreducible. 
 From Lemma \ref{unitary7} it follows that $H\le \SU_7(q^2)$. 
Let $M$ be a maximal subgroup of $\SU_7(q^2)$ which contains $H$. 
Since $H$  is absolutely irreducible,  $M\not\in\mathcal{C}_1\cup\mathcal{C}_3$.
Moreover $M$ cannot belong to $\mathcal{C}_2$  by  Lemma \ref{monomial7}.
Since $z^6$ is not scalar by Lemma \ref{diag}, we may apply \cite[Lemma 2.3]{35}  to deduce that $M\notin\mathcal{C}_6$.
We have $\F_p[a^{7q}]=\F_p[a^7]=\F_{q^2}$, by (iii). 
Since $a^q-1$ is  a coefficient of the characteristic polynomial of $z$, the matrix $z$
cannot be conjugate 
to any element of $\SL_7(q_0)Z$, with $q_0< q^2$ and $Z$ the center of $\SU_7(q^ 2)$.
Thus $M\not\in \mathcal{C}_5$. 
Finally,  $M$ cannot be in class $\mathcal{S}$ by Lemma \ref{class S7} (recall that we are taking $a\neq a^ q$).
We conclude that $H=\SU_7(q^2)$.

As to the existence of some $a$ satisfying all the assumptions when $q>2$, any element
of $\F_{q^2}^ \ast$ of order $q^2-1$  satisfies (iii) by Lemma 2.3 of \cite{35}. The
elements in $\F_{q^2}^\ast$ which do not satisfy either (i) or (ii)  are
at most  $3q+2q=5q$.  If $q \geq 16$,  there are at least 
$5q+1$ elements in $\F_{q^2}^ \ast$ having order $q^2-1$ (use \cite[Lemma 2.1]{35} for
$q\geq 127$ and direct computations otherwise), whence the existence of
$a$.
For $3\leq q\leq 13$, we may take $a$ with minimum polynomial $m_a(t)$
over $\F_p$ as in the following table.

\begin{center}
\begin{tabular}{cc|cc|cc}
$q$ & $m_a(t)$ & $q$ & $m_a(t)$ &$q$ & $m_a(t)$   \\\hline
$3,13$ & $t^2+t-2$ & $4,9$ & $t^4+t^3-1$  & $5,11$ & $t^2+t-3$  \\
$7$ & $t^2+t+3$ & $8$ & $t^6+t+1$\\
\end{tabular}
\end{center}
Then $a$ satisfies (i), (ii) and (iii) and hence $H=\SU_7(q^2)$ for all $q> 2$.
\end{proof}

We now prove the $(2,3)$-generation of $\Omega_7(q)$ for $q$ odd.

\begin{lemma}\label{inOm}
Assume hypothesis \eqref{O} with $q$ odd. Then $H$ is  contained in $\Omega_7(q)$ if and
only if $a-1$ is a square in $\F_q^\ast$.
\end{lemma}

\begin{proof}
In Proposition \ref{unitary7}, we already proved that $H\leq \SO_7(q)$. Furthermore, since
$y=y^{-2}$, we have that $H\leq \Omega_7(q)$ if and only if $x\in \Omega_7(q)$ if and only
if the spinor norm of $x$ is a square in $\F_q^\ast$ (for example, see \cite[Theorem 11.51]{Ta}).
Set $V_x=Im(x-Id)$. It is easy to see that $\{e_1-e_2, e_3-e_4, e_5-e_6,
-a(e_1+e_2)+e_5+e_6+2e_7 \}$ is a basis of $V_x$. The Wall form of $x$ (see \cite[page 153]{Ta}) with respect to this
basis is given by the following matrix 
$$\begin{pmatrix}
-4 & 0 & 2a & 4-4a \\ 0 & -2a & 0 & 4-2a \\ 2a & 0 & -4 & 2a^2-2a \\ 4-4a & 4-2a &
2a^2-2a & 4a-4-2a^2  
  \end{pmatrix}
$$
whose determinant is $16(a-2)^ 2 (a-1) ( a+2)^ 2$. By \cite[page 163]{Ta} $x \in
\Omega_7(q)$ if and only if $a-1$ is a square in $\F_q^\ast$.
\end{proof}

\begin{lemma}\label{SG2}
Assume  hypothesis \eqref{O} and  $q$ odd. If $a\neq 1,2$ then $H$ is not contained in a subgroup
isomorphic to $G_2(q)$.
\end{lemma}

\begin{proof}
Since we are assuming $b=a$, the characteristic polynomial of $z=xy$ is 
\begin{equation}\label{charOm}
\chi_{z}(t)=t^7-t^5-(a-1)t^ 4+(a-1)t^3+t^2-1,
\end{equation}
and the characteristic polynomial of $w=[x,y]$ is 
\begin{equation}\label{commOm}
\chi_{w}(t)=t^7+t^6+t^5-(a^2-4a+3)t^4+(a^2-4a+3)t^3-t^2-t-1.
\end{equation}
The discriminants of $\chi_z(t)$ and $\chi_w(t)$ are, respectively,
$(a-1)(a-5)^ 3 (27 a^2-4a-148)^ 2$ and $(a-2)^6(a+2)^3(a-6)^3(27a^2-108a+76)^2$.
Suppose that $H$ is contained in a subgroup $M\cong G_2(q)$.
Hence we may embed $H$ in $\mathfrak{G}=G_2(\F)$.

Assume first that $a\neq 5$ and $27a^2-4a-148\neq 0$.
Then, the eigenvalues of $z$ are pairwise distinct and so $z$ is a semisimple element of
$\mathfrak{G}$, that is, it belongs to a maximal torus  of $\mathfrak{G}$. By  \cite{G2}
$z$ is conjugate to $s=\diag(1,\alpha,\beta,\alpha \beta,\alpha^ {-1},\beta^ {-1},
(\alpha\beta)^{-1})$ where $\alpha,\beta\in \F^\ast$.
Let $\chi_s(t)=t^ 7+\sum_{j=1}^ 6 f_j t^ j-1$ be the characteristic polynomial of $s$.
It is easy to see that $f_3=f_1+f_1^ 2+f_2$. Comparison with \eqref{charOm} gives
$(a-1)=0+0^ 2 +1$, in contrast with the hypothesis $a\neq 2$.

Consider now the case $a=5$ (and so $p\neq 3$). For this value of $a$, 
the discriminant of $\chi_w(t)$ is
$-3^6 7^3 211^2$ and so, if $p\neq 7,211$, $w$ is a semisimple element of
$\mathfrak{G}$. Proceeding as for $z$, we get  
the contradiction $8=-1+(-1)^2+(-1)$.
If $p=7,211$ we consider the element $wz$.
We get $\chi_{wz}(t)= t^7+3 t^6-16 t^5-82 t^4+82 t^3+16 t^2-3 t-1$, whose discriminant is
non-zero. Hence $wz$ is a semisimple element of
$\mathfrak{G}$, but we obtain the contradiction $82=-3+(-3)^2+16$.

Next, suppose that $a$ is a root of $27a^2-4a-148$. We show that for these values of $a$, the discriminant of $\chi_w(t)$ is non-zero, except when $p=53$ and $a=-6$.
First, notice that for $a=-2$  we get the contradiction $p=2$ and  $a=6$ implies $p=5$, whence $a=1$.
So, assume that $a$ is a root of $27a^2-108a+76$. This gives the conditions $p=5, 53$ and $a=\frac{28}{13}$.
However, for $p=5$ we get the contradiction $a=1$.

Assume $p=53$ and $a=47$. In this case, 
$\chi_{wz^2}(t)=t^7+14 t^6+ 25 t^5+ 6t^4 -6 t^3-25 t^2 -14 t-1$
has non-zero discriminant.
Hence, $wz^2$ is a semisimple element of
$\mathfrak{G}$ but this gives  the contradiction $-6=-14+(-14)^2-25$.
In the other cases, $w$ is a semisimple element of $\mathfrak{G}$ and this produces the contradiction
$(a^2-4a+3)=-1+(-1)^2-1$, that is, $a=2$.
\end{proof}

\begin{proposition}\label{Om35}
The groups $\Omega_7(3)$ and $\Omega_7(5)$ are $(2,3)$-generated. 
\end{proposition}

\begin{proof}
It suffices to take the following matrices of $\SL_7(p)$ for $p=3,5$:
$$
x=\left(\begin{array}{ccccccc}
     0 &     0 &    0 &    1 &    0 &    0 &    0 \\
       0 &    0 &    0 &    0 &    1 &    0 &    7/2 \\
       0 &    0 &    0 &    0 &    0 &    1 & -1/2 \\
       1 &    0 &    0 &    0 &    0 &    0 &    0 \\
       0 &    1 &    0 &    0 &    0 &    0 &    7/2 \\
       0 &    0 &    1 &    0 &    0 &    0 & -1/2 \\
       0 &    0 &    0 &    0 &    0 &    0 &   -1 \\ 
\end{array}\right), \quad
y=\left(\begin{array}{ccccccc}
   1 &    0 &    0 &    0 &    0 &    0 &    0 \\
      0 &    1 &    0 &    0 &    0 &    0 &  7 \\
    0 &    0 &    1 &    1 &   -1 &    0 &    0 \\
     0 &    0 &    0 &    0 &   -1 &    0 &    0 \\
      0 &    0 &    0 &    1 &   -1 &    0 &    0 \\
      0 &    0 &    0 &    0 &    0 &    0 &   -1 \\
      0 &    0 &    0 &    0 &    0 &    1 &   -1 \\
        \end{array}\right).
$$
To prove that $H=\langle x,y\rangle=\Omega_7(p)$, proceed as in
Proposition \ref{U72}.
\end{proof}

\begin{theorem}\label{mainOm7}
Let $x,y$ as in  \eqref{gen7} with $b=a\in \F_{q}$ and $q$ odd.
Suppose that
\begin{itemize}
\item[{\rm(i)}] $a\not \in \{0,1, \pm 2\}$;
\item[{\rm(ii)}] $a-1$ is a square in $\F_q ^ \ast$;
\item[{\rm (iii)}] $\F_{q}=\F_p[a]$.
\end{itemize}
Then $H=\langle x,y\rangle=\Omega_7(q)$. Moreover, if $q\geq 7$, then there exists
$a\in \F_{q}$ satisfying conditions {\rm (i)} to {\rm (iii)} and the
groups $\Omega_7(q)$  are $(2,3)$-generated.
\end{theorem}

\begin{proof}
By Condition (i) and Proposition \ref{irr7} the subgroup $H$ is absolutely irreducible. 
From Lemma \ref{unitary7} it follows that $H\le \SO_7(q)$ and by Condition (ii) and Lemma \ref{inOm}, $H\leq \Omega_7(q)$. 
Let $M$ be a maximal subgroup of $\Omega_7(q)$ which contains $H$. 
Since $H$  is absolutely irreducible,  $M\not\in\mathcal{C}_1$.
Moreover $M$ cannot belong to $\mathcal{C}_2$  by  Lemma \ref{monomial7}. 
Also, by (iii) we have $\F_p[a]=\F_{q}$ and since $a-1$ is  a coefficient of the
characteristic polynomial of $xy$ (see \eqref{charOm}), 
the matrix $xy$ cannot be conjugate to any element of $\SO_7(q_0)$, with $q_0< q$. 
Thus $M\not\in \mathcal{C}_5$. Finally,  $M$ cannot be in class $\mathcal{S}$.
Indeed $H$ cannot be contained in a subgroup isomorphic to $\Sp_6(2)$, by Lemma \ref{S62}, or isomorphic to $G_2(q)$, by Condition (i) and  Lemma \ref{SG2}.
We conclude that $H=\Omega_7(q)$.

The existence of some $a$ satisfying all the assumptions (when $q>5$) is quite clear (take $a=\alpha^2+1$ for some
suitable $\alpha\in \F_q^\ast$ of order $q-1$).
\end{proof}

\begin{remark}
{\rm Taking $b=0$ in \eqref{gen7} and proceeding as in the proof of  previous theorems, for all $q\geq 2$ it is always possible to find a value of $a\in \F_q$ such that the corresponding group $H$ is $\SL_7(q)$. In other words, the elements $x,y$ of \eqref{gen7} are uniform $(2,3)$-generators for all classical groups of dimension $7$.}
\end{remark}

\end{document}